\documentclass[12pt,a4paper]{amsart}
\usepackage[a4paper,margin=24mm,top=25mm,bottom=23mm]{geometry}
\usepackage[T1]{fontenc}
\usepackage[utf8]{inputenc}
\usepackage{amsmath,amssymb,amsfonts,amsthm}
\usepackage{newtxtext,newtxmath}
\usepackage{booktabs}
\usepackage{graphicx}
\usepackage{xcolor}
\usepackage[hidelinks]{hyperref}
\providecommand{\ShowRevisions}{0}
\ifnum\ShowRevisions=1\relax
  
\else
  
\fi
\DeclareRobustCommand{\rev}[1]{%
  \ifnum\ShowRevisions=1\relax
    \textcolor{red}{\hypersetup{linkcolor=red,citecolor=red,urlcolor=red}#1}%
  \else#1\fi}
\numberwithin{equation}{section}
\newtheorem{theorem}{Theorem}[section]
\newtheorem{proposition}[theorem]{Proposition}
\newtheorem{corollary}[theorem]{Corollary}
\theoremstyle{definition}

\newtheorem{example}[theorem]{Example}
\DeclareMathOperator{\Span}{span}
\DeclareMathOperator{\rank}{rank}

\newcommand{\PaperTitle}{Marked Persistent Homology: Circuits and Bordisms}
\newcommand{\FirstAuthor}{\.{I}smail G\"uzel}
\newcommand{\SecondAuthor}{Atabey Kaygun}
\newcommand{\FirstAffiliation}{Department of Computational Mathematics,
Science, and Engineering, Michigan State University,
East Lansing, MI 48823, USA}
\newcommand{\SecondAffiliation}{Department of Mathematical Engineering,
\.{I}stanbul Technical University, Maslak, 34469 \.{I}stanbul, T\"urkiye}
\newcommand{\FirstEmail}{guzelism@msu.edu}
\newcommand{\SecondEmail}{kaygun@itu.edu.tr}
\newcommand{\PaperMSC}{55N31, 05B35, 55N22, 57R90}
\newcommand{\KeywordSeparator}{, }
\newcommand{\PaperKeywords}{Persistent homology\KeywordSeparator
representable matroids\KeywordSeparator circuit ramification\KeywordSeparator
cophenetic ultrametric\KeywordSeparator oriented bordism}

\title{\PaperTitle}
\author{\FirstAuthor}
\address{\FirstAffiliation}
\email{\FirstEmail}
\thanks{\rev{Corresponding author: \FirstAuthor.}}
\author{\SecondAuthor}
\address{\SecondAffiliation}
\email{\SecondEmail}
\subjclass[2020]{\PaperMSC}
\keywords{\PaperKeywords}
\date{}
\begin{document}
\begin{abstract}
The barcode of a persistence module does not determine the relations
among prescribed cycles.  We encode these relations by a filtration
of kernels in fixed coordinates, whose circuits can ramify into
overlapping families and later recombine, unlike the merge-tree
structure of degree-zero persistence.  The kernels recover the marked
image module and admit simultaneous rational bordism realizations.
An explicit planar Vietoris--Rips filtration exhibits the circuit
changes and their surface fillings.

\end{abstract}
\maketitle
\section{Introduction}\label{sec:intro}

The passage from a filtered space to its persistence barcode is
intentionally forgetful.  It retains the isomorphism class of the
resulting persistence module, but suppresses the position of any
specified homology class within that module.  This distinction matters
whenever the cycles themselves carry meaning.  Given several marked
loops, for example, one may ask not only when an individual class
dies, but when a particular linear combination first bounds, whether
the resulting dependence breaks into smaller ones, and whether those
smaller relations subsequently interact.  The barcode does not answer
these questions.  Figure~\ref{fig:example-1} shows two filtrations
with identical degree-one barcodes: in the upper model four marked
loops first satisfy one joint relation, whereas in the middle model
their deaths are independent.  The difference lies in the labelled
relations, not in the abstract persistence modules.

A natural way to retain the lost information is to keep the labels
fixed while allowing the ambient complex to grow.  Choose cycles
$\gamma(e)$ at a scale $s$, indexed by a finite set $E$, and let
$\phi_t:\mathbb Q^E\to H_q(K_t;\mathbb Q)$ send a coordinate vector
to the corresponding class at $t\geq s$.  The increasing family
\[
 R_t=\ker\phi_t\subseteq\mathbb Q^E
\]
records exactly which combinations have become null-homologous.
Restriction to coordinate subspaces defines representable matroids
$M_t$ on the common ground set $E$.  Their circuits detect the
smallest sets of marks that support a relation.  Crucially, a circuit
support is not the relation itself: its coefficients, determined up
to a common scalar, may distinguish marked configurations with the
same underlying matroids.  We therefore retain the embedded
subspaces $R_t$, together with the circuit supports and their
relation lines, rather than replacing them by an unlabelled barcode
or by an abstract sequence of matroids.

This choice has a precise algebraic meaning.  The quotient
$\mathbb Q^E/R_t$ is naturally the space spanned by the marked
classes at time $t$, and these identifications respect the
persistence maps.  Thus the kernel filtration recovers the
\emph{marked image module} with its distinguished coordinate classes
(Theorem~\ref{thm:reconstruction}).  The assertion is an application
of the first isomorphism theorem for persistence modules, not a new
classification theorem; its point is that the fixed presentation
remembers information discarded by passing to the isomorphism class
of the image.  A companion exact sequence
(Proposition~\ref{prop:rank-chain}) distinguishes linear identities
already present among the marked chains from new relations created
when those chains become boundaries.  This distinction prevents a
change in the available representatives from being mistaken for a
genuine change of the original marking.

The geometry of these dependencies diverges sharply from the familiar
hierarchy in degree zero.  When vertices are marked, the only way to
create a relation is to join previously distinct connected
components.  The resulting nested partitions are represented by a
dendrogram.  In positive degrees a minimally dependent family may
instead acquire several proper subcircuits whose supports overlap;
further attachments may produce relations crossing the preceding
branches.  Proposition~\ref{prop:ramification} characterizes the
first phenomenon by a drop in the rank on the original circuit
support, and the explicit linear and simplicial constructions in
Subsection~\ref{sec:circuits} exhibit both overlapping ramification and
later recombination.  There is consequently no rooted splitting
forest which captures the full evolution of circuit incidence.
Figure~\ref{fig:example-2} highlights a necessary qualification:
splitting a cycle into new chains after adding a diagonal also
changes the marking, and is not ramification on the original ground
set.

Could the missing information instead be recovered from pairwise
merger times?  A four-element circuit need not contain any dependent
pair, so already at a single scale its incidence exceeds what can be
read from pairwise comparisons.  Moreover, in higher-degree homology
one must distinguish two nonzero classes becoming proportional from
two classes merely becoming zero.  Treating these events
indiscriminately can destroy the strong triangle inequality.  We
formulate a surviving-merger pseudoultrametric on projective classes
(Proposition~\ref{prop:metric}) and a zero-completed variant; both
record useful pairwise information, but neither determines the
higher-order relations encoded by $R_t$.  The distinction clarifies
how far the cophenetic analogy with hierarchical clustering can be
carried beyond degree zero.

The same relation filtration also admits a geometric reading.  To
say that $\rho\in R_t$ is, at chain level, to say that the prescribed
combination of cycles bounds a simplicial $(q+1)$-chain in $K_t$.
Such fillings provide explicit computational certificates, but they
need not be manifolds.  The classical rational oriented-bordism
comparison supplies a different kind of witness: using a \emph{natural}
section of the orientation map, we choose one family of closed
oriented manifolds mapped into $|K_s|$ whose rational bounding
relations at every later scale are exactly the $R_t$
(Theorem~\ref{thm:geometric}).  Naturality is what preserves the
same labelled relations throughout the filtration; independently
chosen representatives at each scale would not suffice.  Circuits
thereby become minimal supports of mapped bounding relations
(Proposition~\ref{prop:bordism-circuits}).  This is a filtered use of
Thom's theorem, not a claim that simplicial fillings are generally
manifolds or that the bordisms embed in the ambient space.

The planar Vietoris--Rips configuration in
Figure~\ref{fig:vr-realization} makes the distinction tangible.  A
four-loop relation $a+b+c-z=0$ appears when the region between the
outer and inner loops becomes fillable.  Successive fillings of $a$
and $b$ leave the relations $b+c-z=0$ and then $c-z=0$; the latter is
a merger of two nonzero marked classes, not merely their simultaneous
death.  The same generated persistence module admits the alternative
marking $\{a,b,a+b+c-z,z,w\}$, for which all circuits are singletons.
The marked incidence is therefore genuine additional data, yet it
remains sensitive to the chosen marking.  \rev{In this example, the}
integral triangle fillings displayed in
Figure~\ref{fig:vr-realization}(b)--(d) are themselves triangulated
oriented surfaces.  They give explicit mapped bordisms for the
relations, while the numerical rank calculations are justified
independently of the general bordism theorem.

\subsection*{Relation with existing work}

Classical persistence theory answers a different, foundational
question: it describes the isomorphism class of the module by
intervals and controls the stability of the resulting barcode
\cite{EdelsbrunnerLetscherZomorodian02,ZomorodianCarlsson05,
CrawleyBoevey15,CohenSteinerEdelsbrunnerHarer07,
ChazalDeSilvaGlisseOudot16}.  Barcode-derived summaries such as
persistence landscapes and images
\cite{Bubenik15,AdamsEtAl17} necessarily inherit this loss of
marking information.  The categorical status of our reconstruction
is standard: kernels and images are computed pointwise in the
abelian category of persistence modules \cite{BubenikScott14}, and
image persistence is studied computationally by Bauer--Schmahl
\cite{BauerSchmahl25}.  The extra datum here is specifically the
embedding $R_\bullet\subseteq\underline{\mathbb Q^E}$, which names
the coordinates of every relation.  Computing only the barcode of
$\operatorname{im}\phi$ would lose that embedding.

Other approaches retain more geometric information than a barcode,
but make different choices about what is retained.  Matroid methods
in persistence algorithms \cite{HenselmanGhrist16} and optimized
cycle representatives \cite{DeyHouMandal19} address computation and
the selection of cycles.  Harmonic persistent homology
\cite{BasuCox24} and harmonic-chain barcodes
\cite{HouParsaWang25} similarly study distinguished
representatives and their stability.  Generator hypergraphs
\cite{BarbensiEtAl22} and comparisons of analogous bars
\cite{YoonGhristGiusti23} encode relationships among homological
features.  Our question is neither to choose canonical generators
nor to compare barcodes, but to follow the \emph{minimal linear
dependencies of an externally fixed family}.  In particular,
stability of an unmarked barcode alone does not imply stability of
these labelled circuit relations.

\rev{Persistence pairing graphs \cite{Manzanares26} record transitions
among birth cycles selected by a fixed matrix reduction.  Here the
marking is prescribed independently of a reduction, and the relation
filtration retains all linear dependencies and their minimal supports
on that fixed ground set.}

The homological cophenetic construction was introduced in
\cite[Definition~1]{GuzelKaygun22}.  Its connection to hierarchical
clustering is explained by the correspondence between dendrograms
and ultrametrics \cite{CarlssonMemoli10}; the circuit examples above
show the limitation of that correspondence in higher degrees.  The
mapped-bordism realization rests on Thom's rational comparison
\cite{Thom54}, in the precise form
\cite[\S~2.3, Theorem~2.4, p.~29]{KreckLueck05}, and its degree-one
specialization uses the integral comparison
\cite[Theorem~IV.7.37]{Rudyak08}.  For our planar Vietoris--Rips
calculation, the freeness result
\rev{\cite[Corollary~3.2]{ChambersEtAl10} supplies the necessary passage
from finite-field rank certificates to rational homology.}

Unless specified otherwise, simplicial chains and homology groups
have coefficients in $\mathbb Q$.

\section{Marked relations and circuit geometry}\label{sec:relations}

The matroid of cycles in a growing complex has a changing ground set.
For persistent questions about specified cycles, we instead work with
one fixed coordinate space; this makes relations, their supports, and
their changes directly comparable.

\subsection{Cophenetic matroids and quotient reconstruction}

Let $(K_t)_{t\geq0}$ be a filtration of finite simplicial complexes
by inclusions.  At any scale $t$, the quotient representation
$Z_q(K_t)\to H_q(K_t)$ defines a matroid on the set of cycles:
for finite $A\subseteq Z_q(K_t)$, its \emph{cophenetic rank} is
\begin{equation}\label{eq:full-rank}
 c_t^q(A)=\dim\bigl((\Span(A)+B_q(K_t))/B_q(K_t)\bigr).
\end{equation}
Inclusions transport the marked cycles to later scales.

\begin{proposition}\label{prop:cophenetic}
For each $t$, $c_t^q$ is the rank function of a representable
matroid on $Z_q(K_t)$.  If $t\leq u$ and $A\subseteq Z_q(K_t)$ is
finite, then $c_u^q(A)\leq c_t^q(A)$.
\end{proposition}

\begin{proof}
The quotient map $Z_q(K_t)\to H_q(K_t)$ represents the matroid.
For $t\leq u$, its composition with $H_q(K_t)\to H_q(K_u)$
represents the same family at $u$, so ranks cannot increase.
\end{proof}

Since new cycles appear as $t$ grows, the full ground set varies with
the filtration.  A fixed marking is needed to compare circuit supports
across scales.

Fix $s\geq0$, a degree $q\geq0$, a finite set $E$, and
cycles $\gamma(e)\in Z_q(K_s)$ indexed by $e\in E$.  These cycles
remain defined in $K_t$ for $t\geq s$.  Write
\begin{equation*}
 G_t:\mathbb Q^E\longrightarrow Z_q(K_t),\qquad
 \phi_t:\mathbb Q^E\longrightarrow H_q(K_t)
\end{equation*}
for the maps sending $e$ respectively to $\gamma(e)$ and its homology
class.  Repeated cycles and linearly dependent labels are allowed.
For $A\subseteq E$, put
\begin{equation*}
 \begin{aligned}
  S_t(A)&=G_t(\mathbb Q^A),&
  L(A)&=\ker(G_t|_{\mathbb Q^A}),\\
  R_t(A)&=\ker(\phi_t|_{\mathbb Q^A}),&
  r_t(A)&=\dim\phi_t(\mathbb Q^A).
 \end{aligned}
\end{equation*}
The space $L(A)$ is independent of $t$, because the inclusions of
simplicial chain groups are injective.  We abbreviate $R_t(E)$ to
$R_t$.  The functions $r_t$ are the rank functions of the represented
matroids $M_t$; in particular they are submodular \cite{Oxley11}.

\begin{proposition}\label{prop:rank-chain}
For every $A\subseteq E$ and $t\geq s$, there is a short exact sequence
\begin{equation}\label{eq:chain-exact}
 0\longrightarrow L(A)\longrightarrow R_t(A)
 \xrightarrow{G_t}S_t(A)\cap B_q(K_t)\longrightarrow0.
\end{equation}
\end{proposition}

\begin{proof}
A vector $v\in\mathbb Q^A$ belongs to $R_t(A)$ precisely when
$G_t(v)$ is a boundary.  The image of the restricted map is therefore
$S_t(A)\cap B_q(K_t)$, and its kernel is $L(A)$.
\end{proof}

As a consequence, the matroid ranks may be computed directly from
boundary matrices:
\begin{equation}\label{eq:rank}
 r_t(A)=\dim S_t(A)-\dim(S_t(A)\cap B_q(K_t))
 =\rank[\,\partial_{q+1}^t\mid G_t|_A\,]
  -\rank\partial_{q+1}^t.
\end{equation}
Both matrices use the same basis of oriented $q$-simplices.
If the marked chains indexed by $A$ are independent,
$|A|-r_t(A)=\dim(S_t(A)\cap B_q(K_t))$.

Let $h_{t,u}:H_q(K_t)\to H_q(K_u)$ denote the structure map.  The
identity $\phi_u=h_{t,u}\phi_t$ gives $R_t\subseteq R_u$ for
$t\leq u$, while $R_t(A)=R_t\cap\mathbb Q^A$.  Thus a single
increasing filtration of subspaces of $\mathbb Q^E$ controls all
restricted matroids.  Put $V_t^\gamma=\operatorname{im}\phi_t$,
with the induced persistence maps.

\begin{theorem}\label{thm:reconstruction}
The marked cycles induce a short exact sequence of persistence modules
\begin{equation}\label{eq:quotient}
 0\longrightarrow R_\bullet\longrightarrow
 \underline{\mathbb Q^E}\longrightarrow
 V^\gamma_\bullet\longrightarrow0.
\end{equation}
In particular, $V^\gamma_\bullet\cong
(\mathbb Q^E/R_t)_{t\geq s}$, with each distinguished class
$[\gamma(e)]_t$ corresponding to the image of the coordinate vector $e$.
\end{theorem}

\begin{proof}
The maps $\phi_t$ assemble to a morphism
$\phi:\underline{\mathbb Q^E}\to H_q(K_\bullet)$,
since $\phi_u=h_{t,u}\phi_t$.  Kernels and images in the
abelian functor category of persistence modules are evaluated
pointwise \cite[\S~2]{BauerSchmahl25}.  Hence the kernel of
$\underline{\mathbb Q^E}\twoheadrightarrow V^\gamma_\bullet$
is $R_\bullet$, and the first isomorphism theorem gives the
stated coordinate-preserving isomorphism.
\end{proof}

Conversely, an increasing family $R_t\subseteq\mathbb Q^E$
defines a quotient with distinguished coordinate classes.  The
embedding of $R_t$ in the fixed coordinate space matters: the lines
$\mathbb Q(1,1,1)$ and $\mathbb Q(1,1,2)$ have the same circuit
support but yield different marked quotients.  As every relation is a
sum of minimal-support relations, the circuit \emph{coefficient lines},
not merely their supports, recover the fixed marked presentation.
Neither the ambient homology module nor a canonical barcode basis is
reconstructed.

The dimension of $R_t$ increases at most $|E|$ times.  For $A\subseteq B$, the structure maps induce surjections
$\phi_t(\mathbb Q^B)/\phi_t(\mathbb Q^A)\to
\phi_u(\mathbb Q^B)/\phi_u(\mathbb Q^A)$; hence
\[
 r_u(B)-r_u(A)\leq r_t(B)-r_t(A)\qquad(t\leq u).
\]
Thus $M_u$ is a \emph{quotient} of $M_t$: the identity on $E$ is a
strong map of matroids \cite[Proposition~7.3.6]{Oxley11}.
The fixed kernel filtration records more than this matroid perspective,
because it retains the coefficient lines of the dependencies.
If the marks form a basis of $H_q(K_s)$, then
$r_t(E)=\rank h_{s,t}$.

In degree zero, suppose every vertex of $K_s$ is marked by its class.  Then
$r_t(A)$ is the number of components of $K_t$ that meet $A$.
Consequently the $M_t$ are partition matroids, and the nested
partitions form a dendrogram, or a forest when some classes never
merge (compare \cite{BauerBotnanOppermannSteen20}).  In positive
degree, a pairwise dendrogram may still be extracted from a
pseudoultrametric, but it does not determine the circuit incidences.

\subsection{Circuit incidence and ramification}\label{sec:circuits}

A \emph{circuit} of $M_t$ is a minimal nonempty support of a nonzero
vector in $R_t$ \cite[Chapter~1]{Oxley11}.  Such a vector is unique up to scalar: two independent
relations on the same minimal support could be combined to cancel a
coordinate.  Conversely, every vector of $R_t$ is a sum of circuit
relations, by successively subtracting a circuit relation supported
inside its support.  The \emph{circuit incidence graph} $\Gamma_t$
is bipartite, with vertices $E\sqcup\mathcal C_t$ and an edge
$e-C$ whenever $e\in C$.  Each circuit vertex may be decorated by
its relation line $R_t(C)\subseteq\mathbb Q^C$.  Unlike the family
of components in degree zero, the circuits need not be laminar.

For $C\in\mathcal C_t$ and $u>t$, say that $C$ \emph{ramifies} by $u$
when $r_u(C)<r_t(C)$.  Its later circuit incidence consists of the
members of $\mathcal C_u$ contained in $C$.

\begin{proposition}\label{prop:ramification}
A circuit $C\in\mathcal C_t$ ramifies by $u$ if and only if
$C$ contains a proper circuit of $M_u$.
\end{proposition}

\begin{proof}
Since $R_t(C)$ is one-dimensional, ramification means
$\dim R_u(C)>1$.  Two independent relations in $R_u(C)$ can be
combined to cancel a coordinate of the original circuit relation;
the resulting proper support contains a circuit of $M_u$.  Conversely,
a relation on a proper circuit of $M_u$ is independent of the
full-support relation in $R_t(C)$, so $\dim R_u(C)>1$.
\end{proof}

Whether or not it ramifies, $R_u(C)$ is spanned by the later
circuit relations contained in $C$; without ramification, its
unique relation line is unchanged.

\medskip\noindent\textbf{Overlapping circuits.}
Consider the rational matrices
\begin{equation*}
 F=\begin{pmatrix}1&0&0&-1\\0&1&0&-1\\0&0&1&-1\end{pmatrix},
 \qquad T=\begin{pmatrix}1&0&1\\0&1&2\end{pmatrix}.
\end{equation*}
The four columns of $F$, labelled $a,b,c,d$, have the unique
circuit $abcd$.  The matrix $TF$ has rank two and every pair of
its columns is independent; hence its circuits are the four
three-element subsets.  A single circuit therefore ramifies into
overlapping, rather than disjoint, circuits.

\medskip\noindent\textbf{Recombination.}
A different phenomenon occurs for
\begin{equation*}
 F'=\begin{pmatrix}1&0&0&1\\0&1&0&-1\\0&0&1&1\end{pmatrix},
 \quad T'=\begin{pmatrix}1&1&0\\0&0&1\end{pmatrix},
 \quad U'=\begin{pmatrix}1&1\end{pmatrix}.
\end{equation*}
The successive circuit families of $F'$, $T'F'$, and $U'T'F'$
are $\{abcd\}$, $\{ab,cd\}$, and all six two-element subsets,
respectively.  In particular, the final circuit $ac$ crosses
the two intermediate branches, although neither $ab$ nor $cd$
ramifies further.  A rooted splitting forest cannot record
these changes, whereas the incidence graph does.

These representations arise from filtered simplicial complexes.
Start with a finite graph homotopy equivalent to a wedge of three circles,
mark generating cycles by $a,b,c$, and take $d$ to represent
$-a-b-c$ for $F$, or $a-b+c$ for $F'$.  In the first case,
attach a $2$-cell along a loop whose class $-a-2b+c$ spans $\ker T$.
In the second, attach cells successively along $a-b$ and $a-c$.
The cellular boundary map shows that these attachments impose
precisely the indicated relations without creating $H_1$.
To obtain inclusions of finite \emph{simplicial} complexes,
subdivide the attaching circles so that their maps to the preceding
complex are simplicial, adjoin their triangulated mapping cylinders,
and cone off the free ends.  The resulting mapping cones preserve the
preceding complexes as subcomplexes and have the required $H_1$
quotients, giving the stated marked representations up to basis change.

\subsection{Cophenetic distances}\label{sec:metric}

Put $V=H_q(K_s)$ and write $\mathbb P(V)$ for its set of
one-dimensional subspaces.  For $\ell\in\mathbb P(V)$ and $t\geq s$,
write $h_t(\ell)$ for the image under the homology map from $s$ to $t$;
it is either a line or zero.  Passing to lines identifies proportional
nonzero classes, whose distance on the original vector space would
otherwise be zero.  The \emph{surviving cophenetic distance} is
\begin{equation}\label{eq:metric}
 d_s^{\mathrm{surv}}(\ell,m)=
 \inf\{t-s:h_t(\ell)=h_t(m)\ne0\},
 \qquad \inf\varnothing=\infty.
\end{equation}
This is the projective formulation of the nonzero-image merger
condition in \cite[Definition~1]{GuzelKaygun22}; it distinguishes
projective lines instead of treating proportional nonzero vectors as
different points at distance zero.

\begin{proposition}\label{prop:metric}
The surviving cophenetic distance is an extended
pseudoultrametric on $\mathbb P(V)$.
\end{proposition}

\begin{proof}
Only the strong triangle inequality requires verification.
Suppose $\ell$ and $m$ merge nontrivially at $t$ and $m$ and $n$
merge nontrivially at $u$.  If $t\leq u$, the image of $m$ is
nonzero at $u$; since $\ell$ merged with $m$ at $t$, their images
are still equal and nonzero at $u$.  Thus $\ell$ and $n$ merge at
$u$.  The other ordering is symmetric.  Taking infima proves the
inequality.
\end{proof}

The distance may be infinite: two independent classes can die without
ever merging nontrivially.  If the structure maps
are injective on a nontrivial interval immediately after $s$,
$\,d_s^{\mathrm{surv}}$ separates distinct lines.

One may instead declare two lines merged when their images agree
\emph{as subspaces}, allowing both to be zero.  This defines the
\emph{zero-completed} distance
\begin{equation}\label{eq:zero-completed}
 d_s^0(\ell,m)=\inf\{t-s:h_t(\ell)=h_t(m)\}.
\end{equation}
The subspace-equality relations are nested equivalence relations,
so $d_s^0$ is again an extended pseudoultrametric, and
$d_s^0\leq d_s^{\mathrm{surv}}$.  If every homology class dies
by a common finite scale, $d_s^0$ is finite.  This happens for
positive-degree homology of a finite Vietoris--Rips filtration.
\rev{It need not happen on all of $\mathbb P(H_0(K_s))$: a line
of zero augmentation can never merge with a line of nonzero
augmentation.}  For vertex lines in degree zero, however, both
distances agree with the usual cophenetic merge times.

Merely requiring $\dim(h_t(\ell)+h_t(m))<2$ is not a valid
substitute for either convention.  Indeed, let $V=\mathbb Q^3$
and suppose that at time $1$ the first basis vector dies while
the other two remain independent.  The naive condition merges
the first line with each of the other two, but not those two with
each other, violating the ultrametric inequality.  This is why
zero images must be treated explicitly.

For two distinct labels $e,f\in E$ with nonzero initial classes,
the zero-completed merger condition is
$r_t(\{e,f\})=r_t(\{e\})=r_t(\{f\})$;
a surviving merger requires this common rank to be one.

\subsection{Identical barcodes and changes of marking}

The following examples separate the information carried by fixed
circuit relations from both the unmarked barcode and relations that
appear only after the marking itself has changed.

\begin{example}[\rev{Identical barcodes, different circuits}]\label{ex:barcode}

  In the upper panel of Figure~\ref{fig:example-1}, disjoint inner loops $a,b,c$
  appear at time $1$ and an outer loop $z$ at time $2$.
  At time $3$, attach a triangulated four-holed planar region $\Sigma$ along
  these loops, so that $\partial_2\Sigma=a+b+c-z$ and $[z]-[a]-[b]-[c]=0$.
  Fill the inner holes at times $4,5,6$.
  \rev{The class $[z]-[a]-[b]-[c]$, born at time $2$, dies at time
  $3$, whereas the marked class $[z]$ remains nonzero until time
  $6$ and equals $[a]+[b]+[c]$ on $[3,6)$.
  The nonsingleton circuits at times $3,4,5$ are,
  respectively, $\{a,b,c,z\}$, $\{b,c,z\}$, and $\{c,z\}$.
  The vanished marks give singleton circuits; at time $6$ all four
  circuits are singletons.}

  Instead take four independent hollow triangles born at $1,1,1,2$ and
  filled at $4,5,6,3$ (middle panel of Figure~\ref{fig:example-1}).  Both constructions
  share the degree-one barcode $\{[1,4),[1,5),[1,6),[2,3)\}$
  (Figure~\ref{fig:example-1}, bottom), but the second has only singleton
  circuits and no four-element circuit.  Contractible connecting edges may
  make both models connected without changing their first homology.  The
  distinction concerns the markings, not the underlying isomorphism classes of
  persistence modules.
  
  \begin{figure}[ht]
    \centering
    \includegraphics[width=\linewidth]{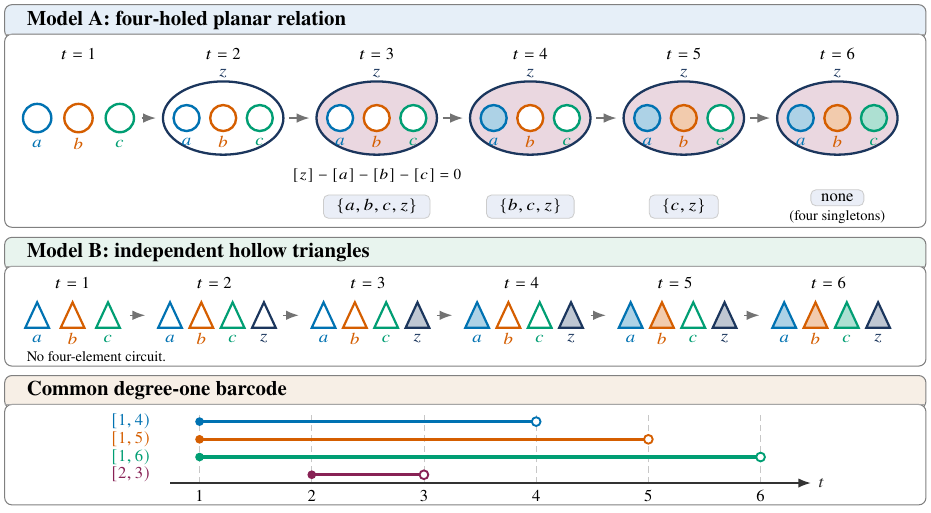}
    \begingroup
    \ifnum\ShowRevisions=1\relax\color{red}\fi
    \caption{Two filtrations with identical unmarked $H_1$ barcodes
    (bottom) and different marked circuits.  The upper labels list
    nonsingleton circuits; only singletons remain at $t=6$.
    In Model~A, the interval $[2,3)$ is represented by
    $[z]-[a]-[b]-[c]$, while $[z]$ survives until $t=6$.
    Model~B has only singleton circuits.}\label{fig:example-1}
    \endgroup
  \end{figure}
\end{example}

\begin{example}[\rev{Splitting at chain level}]\label{ex:chain-split}
  Consider Figure~\ref{fig:example-2}.  Before inserting the diagonal
  (left), the square cycle $z=AB+BC+CD+DA$ forms a single mark $E=\{z\}$.
  If the diagonal $BD$ is inserted without attaching any triangles
  (right), the triangular cycles $a=AB+BD+DA$ and $b=BC+CD+DB$ satisfy
  $z=a+b$ directly as $1$-chains.  Under the enlarged marking
  $E'=\{z,a,b\}$, \rev{the coefficient vector $(1,-1,-1)$ in the
  order $(z,a,b)$ belongs to the chain-level kernel $L(E')$.}
  Moreover, $a$ and $b$ were not
  defined before $BD$ appeared.  This is a change of marking, not
  ramification of a circuit on a fixed ground set.

  \begin{figure}[ht]
    \centering
    \includegraphics[width=\linewidth]{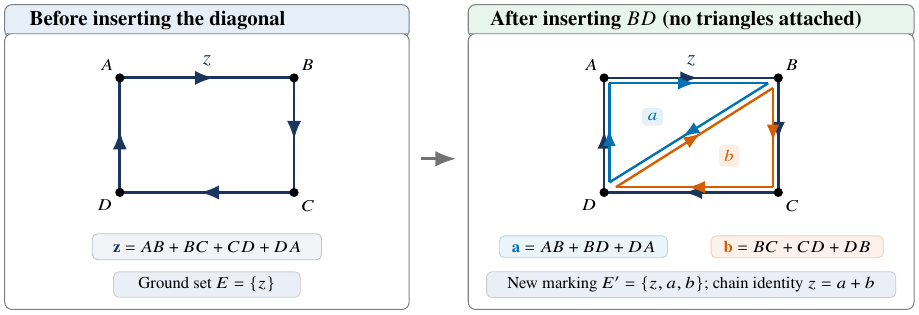}
    \begingroup
    \ifnum\ShowRevisions=1\relax\color{red}\fi
    \caption{Adding the diagonal $BD$ creates the chain identity
    $z=a+b$.  Its coefficient vector $(1,-1,-1)$ in the order
    $(z,a,b)$ belongs to $L(E')$ for the enlarged marking
    $E'=\{z,a,b\}$ (right).  The original marking $E=\{z\}$ (left)
    acquires no new relation.}\label{fig:example-2}
    \endgroup
  \end{figure}
\end{example}

\section{Geometric realization of marked relations}\label{sec:bordism}

We distinguish two kinds of realization.  The matrix computations
produce simplicial $(q+1)$-chains whose boundaries are the marked
relations.  Mapped oriented bordism instead realizes the same
relations as boundaries of manifolds mapped to the filtered spaces.
The latter adds a manifold interpretation; it is not an input to the
chain-level algorithms.

\subsection{Chain fillings}

A coefficient vector $\rho$ belongs to $R_t$ precisely when
$G_t(\rho)=\partial\Sigma$ for some
$\Sigma\in C_{q+1}(K_t;\mathbb Q)$.  We call $\Sigma$ a
\emph{chain filling}; it need not be a manifold or the image of a
manifold's fundamental chain.

The vector space of fillings of relations is
\[
 \mathcal F_t=\{(\rho,\Sigma)\in
 \mathbb Q^E\oplus C_{q+1}(K_t;\mathbb Q):
 \partial\Sigma=G_t(\rho)\}.
\]

\begin{proposition}\label{prop:fillings}
Projection onto the first coordinate gives a short exact sequence
\begin{equation}\label{eq:filling-exact}
 0\longrightarrow Z_{q+1}(K_t;\mathbb Q)\longrightarrow
 \mathcal F_t\longrightarrow R_t\longrightarrow0.
\end{equation}
\end{proposition}

\begin{proof}
The image of the projection consists exactly of those \(\rho\) for
which \(G_t(\rho)\) is a boundary, namely \(R_t\).  The kernel
consists of pairs \((0,\Sigma)\) with \(\partial\Sigma=0\), and is
therefore \(Z_{q+1}(K_t;\mathbb Q)\).
\end{proof}

Thus the fillings of any fixed relation form an affine space over
$Z_{q+1}(K_t;\mathbb Q)$; the zero chain fills each element of $L(E)$.

\subsection{Mapped oriented bordism and simultaneous realization}

Let $X$ be a CW complex.  A \emph{mapped oriented $q$-manifold over
$X$} is a closed smooth oriented $q$-manifold $N$, possibly
disconnected, together with a continuous map $f:N\to X$.  Its
oriented bordism class is denoted $[N,f]\in\Omega_q^{SO}(X)$.
A null-bordism of $(N,f)$ consists of a compact smooth oriented
$(q+1)$-manifold $W$ and a map $F:W\to X$ restricting to $f$ on
$\partial W=N$.  More generally, $[N_0,f_0]=[N_1,f_1]$ if
$(-N_0)\sqcup N_1$ bounds in this sense.  Disjoint union gives the
group operation.  The map $F$ need not be an embedding, and a mapped
bordism need not be a simplicial subcomplex of $X$; see
\cite[\S~2.3, pp.~28--29]{KreckLueck05}.

The orientation homomorphism
\[
 \mu_X:\Omega_q^{SO}(X)\longrightarrow H_q(X;\mathbb Z),
 \qquad [N,f]\longmapsto f_*[N],
\]
will also denote its rationalization.  By Thom's rational bordism
theorem \cite{Thom54}, in the form
\cite[\S~2.3, Theorem~2.4, p.~29]{KreckLueck05}, the normal Gauss
map $\nu_N:N\to BSO$ gives a natural isomorphism
\[
 \Psi_X:\Omega_q^{SO}(X)\otimes\mathbb Q
 \xrightarrow{\ \cong\ } H_q(BSO\times X;\mathbb Q),
 \qquad [N,f]\longmapsto(\nu_N\times f)_*[N].
\]
Under $\Psi_X$, $\mu_X$ is induced by the projection
$BSO\times X\to X$.  The inclusion $i_X(x)=(b,x)$, for a
fixed base point $b\in BSO$, therefore defines a natural section
\begin{equation}\label{eq:section}
 \sigma_X=\Psi_X^{-1}(i_X)_*:\ H_q(X;\mathbb Q)
 \longrightarrow\Omega_q^{SO}(X)\otimes\mathbb Q,
 \qquad \mu_X\sigma_X=\operatorname{id}.
\end{equation}
Its naturality, not merely the existence of individual manifold
representatives, is essential: the same chosen classes must retain
their relations after passage from $X_s$ to every $X_t$.
The rational K\"unneth decomposition and
$H^*(BSO;\mathbb Q)=\mathbb Q[p_1,p_2,\ldots]$, $|p_i|=4i$,
show why $\mu_X$ is generally not an isomorphism in degree $q\geq4$.

Put $X_t=|K_t|$ and let $\iota_{s,t}:X_s\hookrightarrow X_t$.
Any fixed family $(N_e,f_e)_{e\in E}$ over $X_s$ gives maps
\begin{equation}\label{eq:bordismkernel}
 \alpha_t:\mathbb Q^E\longrightarrow
 \Omega_q^{SO}(X_t)\otimes\mathbb Q,
 \qquad e\longmapsto[N_e,\iota_{s,t}f_e]\otimes1.
\end{equation}

\begin{theorem}\label{thm:geometric}
There exist a positive integer $D$ and a fixed family of mapped
closed oriented $q$-manifolds $(N_e,f_e)_{e\in E}$ over $X_s$
such that $(f_e)_*[N_e]=D[\gamma(e)]$ rationally and
$\ker\alpha_t=R_t$ for every $t\geq s$.
\end{theorem}

\begin{proof}
For every $e\in E$, choose an integral bordism class mapping to
$D\sigma_{X_s}([\gamma(e)])$ under rationalization, where one
common positive integer $D$ clears the finitely many denominators.
Represent these integral classes by mapped manifolds $(N_e,f_e)$.
Equation~\eqref{eq:section} gives the asserted fundamental classes.
Naturality gives
\[
 \alpha_t=D\,\sigma_{X_t}\phi_t
 \qquad(t\geq s).
\]
The section $\sigma_{X_t}$ is injective, so
$\ker\alpha_t=\ker\phi_t=R_t$ simultaneously for all $t$.
\end{proof}

For an integer vector $n=(n_e)\in\mathbb Z^E$, let
$\mathscr N_t(n)$ denote the signed disjoint union over $X_t$ of
$|n_e|$ copies of $(N_e,\iota_{s,t}f_e)$, with reversed
orientation when $n_e<0$.  It represents
$\sum_en_e[N_e,\iota_{s,t}f_e]$ in integral oriented bordism.

\begin{proposition}\label{prop:bordism-circuits}
For $t\geq s$ and $\rho\in\mathbb Q^E$, the following are equivalent:
\begin{enumerate}
\item $\rho\in R_t$;
\item $\alpha_t(\rho)=0$;
\item for some positive integer $m$ with
      $m\rho\in\mathbb Z^E$, the mapped manifold
      $\mathscr N_t(m\rho)$ bounds an oriented $(q+1)$-manifold
      over $X_t$.
\end{enumerate}
Consequently, the circuits at $t$ are precisely the minimal
nonempty supports of such rational bounding relations.
\end{proposition}

\begin{proof}
The first two conditions are equivalent by
Theorem~\ref{thm:geometric}.  If $\alpha_t(\rho)=0$, clear
denominators by $d>0$.  The integral class
$[\mathscr N_t(d\rho)]$ has zero rationalization, and is
therefore torsion; multiplying by its order gives (3).
Conversely a null-bordism of $\mathscr N_t(m\rho)$ implies
$m\alpha_t(\rho)=0$, hence $\alpha_t(\rho)=0$ over
$\mathbb Q$.  Minimal supports of nonzero vectors in $R_t$
are exactly the circuits by definition.
\end{proof}

The multiplier $D$ is independent of $t$; the torsion-clearing
multiplier $m$ may depend on $(t,\rho)$.  The result prescribes neither
canonical manifold representatives nor a genus bound or a
comparison with a particular chain filling.

In degrees $q<4$, the integral orientation map
$\Omega_q^{SO}(X)\to H_q(X;\mathbb Z)$ is an isomorphism for CW
complexes \cite[Theorem~IV.7.37]{Rudyak08}.  Consequently,
for the degree-one examples we can start directly with the marked
oriented loops, without a rational choice of $\sigma_X$ or a
scaling factor $D$.

\begin{corollary}\label{cor:surfaces}
Suppose $q=1$ and the marks are oriented loops
$f_e:S^1\to X_s$ representing $[\gamma(e)]$.  For every
$t\geq s$, a vector $\rho\in\mathbb Q^E$ belongs to $R_t$
if and only if some positive integral multiple of its signed
family of mapped loops bounds a compact oriented surface over
$X_t$.
\end{corollary}

\begin{proof}
By the cited low-dimensional comparison, the bordism class of a
signed family of loops vanishes rationally exactly when its
homology class does.  An integral bordism class has zero
rationalization precisely when it is torsion.  Clearing the
coefficients of $\rho$ and then multiplying by its torsion order
gives the required surface; the converse follows by applying
the orientation homomorphism.
\end{proof}

The upper panel of Figure~\ref{fig:example-1} illustrates a
particularly simple case: after the four-holed region is attached,
its oriented surface has boundary $a+b+c-z$, so the circuit is
realized without multiplication.
\rev{The middle panel has no four-element circuit.}
Figure~\ref{fig:example-2} records a different
phenomenon: $z=a+b$ is already an identity of chains after the
marking is enlarged, not a new bordism relation arising from the
death of a class.  In general, even when a loop relation bounds a
mapped surface, the surface need not be a disk or an embedded
subspace of $X_t$; a commutator in a wedge of two circles, for
instance, bounds a mapped punctured torus but no mapped disk.

\section{Computing and realizing circuit ramification}\label{sec:computation}

We now make the relation filtration explicit for a finite filtration.
The planar Vietoris--Rips configuration then shows which circuit
changes are detected by the computation, and which of their chain
fillings also give mapped bordisms.

\subsection{Detecting changes in the relation filtration}

The subspace filtration $(R_t)_{t\geq s}$ and the chain fillings of
Proposition~\ref{prop:fillings} can be computed over $\mathbb Q$
directly from the boundary matrices of a finite filtration.  Let
$t_0=s<t_1<\dots<t_N$ be the distinct filtration values at which
$(q+1)$-simplices enter $K_t$ after scale $s$.  Evaluating
\eqref{eq:rank} naively at every $t_k$ would require $N$ boundary
reductions.  However, because $R_t\subseteq R_u$ for $t\leq u$ and
$R_t(A)=R_t\cap\mathbb Q^A$, equality $r_t(E)=r_u(E)$ implies
$R_t=R_u$ and hence constancy of all restricted matroids and circuit
families between $t$ and $u$.  Since $r_t(E)$ is nonincreasing and takes values between $0$ and
$|E|$, evaluating its value at $t_N$ and using binary search for the
first index satisfying $r_{t_k}(E)\leq r_s(E)-j$, for each attained
threshold $j$, locates all changes of $R_t$ using at most
$1+|E|\lceil\log_2(N+1)\rceil$ evaluations of \eqref{eq:rank}
(after $r_s(E)$ is known).

At each critical scale, rational elimination and minimal-support
enumeration determine:
\begin{enumerate}
 \item the time-independent chain-level kernel $L(E)=\ker G_s$ and the
       boundary dimension $\dim(S_t(E)\cap B_q(K_t))=\dim R_t-\dim L(E)$
       in the exact sequence \eqref{eq:chain-exact};
 \item the active circuits $C\in\mathcal C_t$ and their relation
       lines $R_t(C)=\Span(\rho_C)\subseteq\mathbb Q^C$; and
 \item an explicit chain filling $\Sigma_C\in C_{q+1}(K_t;\mathbb Q)$
       satisfying $\partial_{q+1}^t\Sigma_C=G_t(\rho_C)$ in the exact
       sequence \eqref{eq:filling-exact}.
\end{enumerate}
Enumeration of all circuits can require exponential output size in
$|E|$; the bisection bound concerns the number of rank evaluations,
not the complexity of circuit enumeration or exact elimination.
\rev{For Euclidean Vietoris--Rips filtrations, critical parameters
are edge lengths; their displayed decimal values below are rounded.}

The next decomposition is pointwise in the filtration parameter;
its blocks need not persist.  Let $G$ be the $1$-skeleton of a finite
simplicial complex $K$.
For each biconnected block $G_j$ containing a graph cycle, let
$K_j$ consist of its edges, vertices, and the $2$-simplices of
$K$ whose edges lie in $G_j$.

\begin{proposition}\label{prop:blocks}
The inclusions induce an isomorphism
$\bigoplus_j H_1(K_j)\cong H_1(K)$.
\end{proposition}

\begin{proof}
Graph cycles are sums of simple cycles, each supported in one
biconnected block.  Because the blocks have disjoint edge sets,
$Z_1(K)=\bigoplus_j Z_1(K_j)$.  Every $2$-simplex boundary lies
in one block, so likewise
$B_1(K)=\bigoplus_j B_1(K_j)$.  Taking quotients proves the claim.
\end{proof}

If each marked cycle is supported in one of these blocks, its
represented matroid is the direct sum of the corresponding blockwise
matroids.  The hypothesis must be checked anew at each scale: an
edge that is initially a bridge may cease to be a bridge later.

\subsection{A planar Vietoris--Rips realization}

\begin{figure}[ht]
  \centering
  \includegraphics[width=\linewidth]{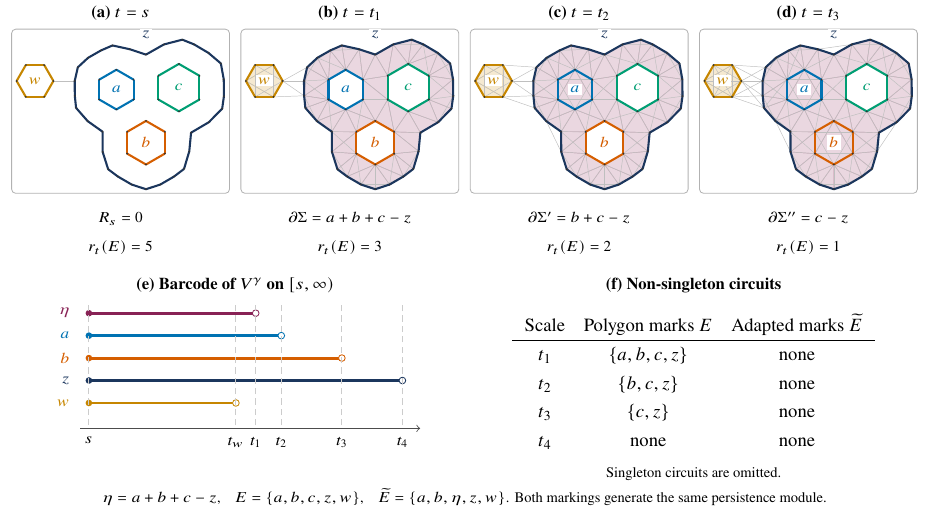}
  \caption{Marked relations in a planar Euclidean Vietoris--Rips
  \rev{filtration on $54$ vertices.}
  \textbf{(a)--(d)} The initial polygon marking and selected
  oriented triangle fillings.  The fillings in (b)--(d) are
  triangulated mapped surfaces with four, three, and two boundary
  components, respectively.  \textbf{(e)} An adapted barcode of the marked
  image module: the class $\eta=a+b+c-z$ dies at $t_1$, while the
  specified class $c$ remains nonzero until $t_4$.
  \textbf{(f)} Circuit supports for the polygon marking $E$ versus
  the basis-adapted marking $\widetilde E=\{a,b,\eta,z,w\}$, which
  has only singleton circuits.}\label{fig:vr-realization}
\end{figure}

\begin{example}[Euclidean Vietoris--Rips realization]\label{ex:vr-realization}
\rev{Consider the Vietoris--Rips filtration $K_t=\mathrm{VR}_t(D)$
of the $54$-point planar configuration in
Figure~\ref{fig:vr-realization}.  At the marking scale $s=0.98$,
the five vertex-disjoint polygon cycles $E=\{a,b,c,z,w\}$
consist of three inner hexagons $a,b,c$, a $30$-vertex loop $z$
surrounding them, and a hexagon $w$ attached to $z$ by a single
bridge edge.}

Because the five marked $1$-chains have disjoint edge supports and
$K_s$ contains no $2$-simplices, Proposition~\ref{prop:rank-chain}
gives $L(E)=\{0\}$, $R_s=\{0\}$, $r_s(E)=5$, and $\mathcal C_s=\varnothing$
(Figure~\ref{fig:vr-realization}(a)).
\rev{The marked rank drops at five scales $t_w<t_1<t_2<t_3<t_4$
(Figure~\ref{fig:vr-realization}(e)--(f) and
Table~\ref{tab:vr-cascade}):}
\begin{enumerate}
 \item \emph{Death of $w$ ($t_w\approx1.32$).}
       The initial bridge edge ceases to be a bridge before $t_w$;
       Proposition~\ref{prop:blocks} therefore does not apply
       uniformly across the filtration.  At $t_w\approx1.32$, four
       $2$-simplices $\Sigma_w$ fill $w$ ($\partial_2\Sigma_w=w$),
       creating the singleton circuit $\{w\}$ ($r_{t_w}(\{w\})=0$).
       \rev{Nevertheless, the relation filtration satisfies
       $R_t(E)=R_t(\{a,b,c,z\})\oplus R_t(\{w\})$ at every
       scale.}  The resulting matroid direct sum is an algebraic fact
       in this example, not a consequence of persistent graph blocks.
 \item \emph{Four-holed planar filling ($t_1\approx1.36$).}
       While the interiors of $a,b,c$ remain empty of chords, $52$
       triangles in $C_2(K_{t_1})$ triangulate the region between $z$
       and $a,b,c$ (Figure~\ref{fig:vr-realization}(b)), producing a
       chain filling $\Sigma$ with
       $\partial_2\Sigma=a+b+c-z\in S_{t_1}(E)\cap B_1(K_{t_1})$.
       This introduces the four-element circuit $C=\{a,b,c,z\}$ of
       rank $3$ with relation $[a]+[b]+[c]-[z]=0$
       (Figure~\ref{fig:vr-realization}(f), top).  In the barcode of the module restricted to $t\geq s$
       (Figure~\ref{fig:vr-realization}(e)), this event terminates
       the interval represented by $\eta=a+b+c-z$; it does not kill
       the individual class $c$.  The barcode alone does not specify
       the relation among $a,b,c,z$.
 \item \emph{Ramification cascade ($t_2\approx1.42$, $t_3\approx1.56$, and $t_4\approx1.70$).}
       At $t_2\approx1.42$, four triangles $\Sigma_a$ cap $a$
       ($\partial_2\Sigma_a=a$, so $[a]=0$), and $\{a,b,c,z\}$ ramifies
       as in Proposition~\ref{prop:ramification} into $\{a\}$ and the
       three-element circuit $\{b,c,z\}$ of rank $2$, filled by the
       $56$-triangle chain $\Sigma'=\Sigma-\Sigma_a$ with
       $\partial_2\Sigma'=b+c-z$ (Figure~\ref{fig:vr-realization}(c),(f)).
       At $t_3\approx1.56$, four triangles $\Sigma_b$ cap $b$
       ($\partial_2\Sigma_b=b$, so $[b]=0$), and $\{b,c,z\}$ ramifies
       into $\{b\}$ and the surviving merger $\{c,z\}$ of rank $1$,
       filled by the $60$-triangle chain $\Sigma''=\Sigma'-\Sigma_b$
       with $\partial_2\Sigma''=c-z$ (Figure~\ref{fig:vr-realization}(d),(f)).
       \rev{The marked class $[c]=[z]\neq0$ persists across $[t_3,t_4)$}
       (Figure~\ref{fig:vr-realization}(e)) until $t_4\approx1.70$, when four
       triangles $\Sigma_c$ cap $c$ and $\{c,z\}$ splits into the
       singletons $\{c\}$ ($4$ triangles) and $\{z\}$ ($64$ triangles;
       Figure~\ref{fig:vr-realization}(f), bottom).
\end{enumerate}

For comparison, consider the alternative marking
$\widetilde E=\{a,b,\eta,z,w\}$, where
$\eta=a+b+c-z$ is the \emph{chain} sum at scale $s$.  Since
$c=z-a-b+\eta$, both markings generate the same persistence module.
The adapted marks die separately, giving only singleton circuits
(Figure~\ref{fig:vr-realization}(f)).  The circuit structure
therefore depends on the marking even when the generated module
and its barcode do not.

Table~\ref{tab:vr-ultrametric} compares the resulting merge scales
$t^{\mathrm{surv}}=s+d_s^{\mathrm{surv}}$ and $t^0=s+d_s^0$ on $E$.
Only the pair $(c,z)$ undergoes a surviving merger, at
$t^{\mathrm{surv}}(c,z)\approx1.56$ ($d_s^{\mathrm{surv}}(c,z)\approx0.58$).
All other pairs die without merging nontrivially, so their surviving
cophenetic distances are infinite while their zero-completed merge
scales equal the maximum of their individual death scales.
\end{example}

\subsection{Surface fillings and exact verification}
\rev{The polygonal marks fall under Corollary~\ref{cor:surfaces}.
The fillings in Figure~\ref{fig:vr-realization}(b)--(d) are connected,
oriented surfaces with four, three, and two boundary components,
representing $a+b+c-z$, $b+c-z$, and $c-z$, respectively.
Each maps by inclusion into the corresponding Vietoris--Rips complex,
so these are integral mapped bordisms.  Capping all three inner
loops gives a disk filling $z$.  The circuits identify the minimal
bounding supports, while their relation lines determine the
coefficients up to a common nonzero factor.  Theorem~\ref{thm:geometric}
provides a rational bordism realization in all degrees.}

\rev{At $s,t_w,t_1,t_2,t_3,t_4$, reduction modulo a suitable prime
gives marked ranks $5,4,3,2,1,0$.  Planar Vietoris--Rips complexes
have free fundamental group \cite[Corollary~3.2]{ChambersEtAl10},
so $H_1(K_t;\mathbb Z)$ is free abelian.  Reduction of the integral
marked classes therefore gives lower bounds for their rational ranks;
the integral fillings above give matching upper bounds.
Between consecutive critical scales, constancy of the modular rank
and monotonicity of the rational rank exclude additional drops.
Thus the relations in Table~\ref{tab:vr-cascade} generate $R_t$
throughout the filtration.  The filling sizes count triangles
in the exhibited surfaces; no minimality is asserted.}

\begin{table}[ht]
  \centering
  \caption{Exact-sequence dimensions (Proposition~\ref{prop:rank-chain}),
  active circuits $\mathcal C_t$, relation generators in $R_t$, and
  chain-filling sizes $|\Sigma_C|$ (Proposition~\ref{prop:fillings})
  for the Vietoris--Rips filtration illustrated in Figure~\ref{fig:vr-realization}
  at fixed marking scale \mbox{\rev{$s=0.98$.}}}\label{tab:vr-cascade}
  \footnotesize
  \renewcommand{\arraystretch}{1.15}
  \setlength{\tabcolsep}{4.5pt}
  \begin{tabular}{@{}lcccccc@{}}
    \toprule
    Scale $t$ & $r_t(E)$ & $\dim L(E)$ & $\dim(S_t\cap B_1)$ & Active circuits $\mathcal C_t$ in $\Gamma_t$ & New relation in $R_t$ & $|\Sigma_C|$ \\
    \midrule
    $s = 0.98$   & $5$ & $0$ & $0$ & $\varnothing$ & $R_s = \{0\}$ & $0$ \\
    $t_w\approx 1.32$ & $4$ & $0$ & $1$ & $\{w\}$ & $[w] = 0$ & $4$ \\
    $t_1\approx 1.36$ & $3$ & $0$ & $2$ & $\{w\},\; \{a,b,c,z\}$ & $[a]+[b]+[c]-[z]=0$ & $52$ \\
    $t_2\approx 1.42$ & $2$ & $0$ & $3$ & $\{w\},\; \{a\},\; \{b,c,z\}$ & $[a]=0,\; [b]+[c]-[z]=0$ & $56$ \\
    $t_3\approx 1.56$ & $1$ & $0$ & $4$ & $\{w\},\; \{a\},\; \{b\},\; \{c,z\}$ & $[b]=0,\; [c]-[z]=0$ & $60$ \\
    $t_4\approx 1.70$ & $0$ & $0$ & $5$ & $\{w\},\; \{a\},\; \{b\},\; \{c\},\; \{z\}$ & $[c]=0,\; [z]=0$ & $64$ \\
    \bottomrule
  \end{tabular}
\end{table}

\begin{table}[ht]
  \centering
  \caption{Surviving merge scales $t^{\mathrm{surv}}=s+d_s^{\mathrm{surv}}$
  versus zero-completed merge scales $t^0=s+d_s^0$
  (displayed as $t^{\mathrm{surv}}/t^0$) on $E=\{a,b,c,z,w\}$ at
  \mbox{\rev{$s=0.98$.}}}\label{tab:vr-ultrametric}
  \footnotesize
  \renewcommand{\arraystretch}{1.15}
  \setlength{\tabcolsep}{6.0pt}
  \begin{tabular}{@{}lccccc@{}}
    \toprule
    Mark & $a$ & $b$ & $c$ & $z$ & $w$ \\
    \midrule
    $a$ & $0.98 / 0.98$ & $\infty / 1.56$ & $\infty / 1.70$ & $\infty / 1.70$ & $\infty / 1.42$ \\
    $b$ & $\infty / 1.56$ & $0.98 / 0.98$ & $\infty / 1.70$ & $\infty / 1.70$ & $\infty / 1.56$ \\
    $c$ & $\infty / 1.70$ & $\infty / 1.70$ & $0.98 / 0.98$ & $\mathbf{1.56 / 1.56}$ & $\infty / 1.70$ \\
    $z$ & $\infty / 1.70$ & $\infty / 1.70$ & $\mathbf{1.56 / 1.56}$ & $0.98 / 0.98$ & $\infty / 1.70$ \\
    $w$ & $\infty / 1.42$ & $\infty / 1.56$ & $\infty / 1.70$ & $\infty / 1.70$ & $0.98 / 0.98$ \\
    \bottomrule
  \end{tabular}
\end{table}

\section*{\rev{Funding}}
\rev{This research did not receive any specific grant from funding
agencies in the public, commercial, or not-for-profit sectors.}

\section*{\rev{Declaration of competing interest}}
\rev{The authors declare that they have no known competing financial
interests or personal relationships that could have appeared to
influence the work reported in this paper.}

\section*{\rev{Declaration of generative AI and AI-assisted technologies
in the manuscript preparation process}}
\rev{During the preparation of this work, OpenAI Codex (GPT-6, OpenAI)
was used for language revision and figure preparation.
The authors are responsible for the mathematical
arguments, verification, and final content of the article.}


\begin{thebibliography}{99}
\bibitem{EdelsbrunnerLetscherZomorodian02}
H. Edelsbrunner, D. Letscher, and A. Zomorodian,
Topological persistence and simplification,
\emph{Discrete Comput. Geom.} \textbf{28} (2002), 511--533,
\href{https://doi.org/10.1007/s00454-002-2885-2}{doi:10.1007/s00454-002-2885-2}.
\bibitem{ZomorodianCarlsson05}
A. Zomorodian and G. Carlsson,
Computing persistent homology,
\emph{Discrete Comput. Geom.} \textbf{33} (2005), 249--274,
\href{https://doi.org/10.1007/s00454-004-1146-y}{doi:10.1007/s00454-004-1146-y}.
\bibitem{CrawleyBoevey15}
W. Crawley-Boevey, Decomposition of pointwise finite-dimensional persistence
modules, \emph{J. Algebra Appl.} \textbf{14} (2015), 1550066,
\href{https://doi.org/10.1142/S0219498815500668}{doi:10.1142/S0219498815500668}.
\bibitem{CohenSteinerEdelsbrunnerHarer07}
D. Cohen-Steiner, H. Edelsbrunner, and J. Harer,
Stability of persistence diagrams,
\emph{Discrete Comput. Geom.} \textbf{37} (2007), 103--120,
\href{https://doi.org/10.1007/s00454-006-1276-5}{doi:10.1007/s00454-006-1276-5}.
\bibitem{ChazalDeSilvaGlisseOudot16}
F. Chazal, V. de Silva, M. Glisse, and S. Oudot,
\emph{The Structure and Stability of Persistence Modules},
SpringerBriefs in Mathematics, Springer, 2016,
\href{https://doi.org/10.1007/978-3-319-42545-0}{doi:10.1007/978-3-319-42545-0}.
\bibitem{Bubenik15}
P. Bubenik, Statistical topological data analysis using persistence landscapes,
\emph{J. Mach. Learn. Res.} \textbf{16} (2015), 77--102,
\url{https://jmlr.org/papers/v16/bubenik15a.html}.
\bibitem{AdamsEtAl17}
H. Adams, T. Emerson, M. Kirby, R. Neville, C. Peterson, P. Shipman,
S. Chepushtanova, E. Hanson, F. Motta, and L. Ziegelmeier,
Persistence images: A stable vector representation of persistent homology,
\emph{J. Mach. Learn. Res.} \textbf{18} (2017), no.~8, 1--35,
\url{https://www.jmlr.org/papers/v18/16-337.html}.
\bibitem{BubenikScott14}
P. Bubenik and J. A. Scott,
Categorification of persistent homology,
\emph{Discrete Comput. Geom.} \textbf{51} (2014), no.~3,
600--627,
\href{https://doi.org/10.1007/s00454-014-9573-x}{doi:10.1007/s00454-014-9573-x}.
\bibitem{BauerSchmahl25}
U. Bauer and M. Schmahl,
Efficient computation of image persistence,
\emph{Discrete Comput. Geom.} \textbf{74} (2025), 999--1019,
\href{https://doi.org/10.1007/s00454-025-00769-8}
{doi:10.1007/s00454-025-00769-8}.
\bibitem{HenselmanGhrist16}
G. Henselman and R. Ghrist,
Matroid filtrations and computational persistent homology,
\emph{arXiv:1606.00199}, 2016,
\rev{\href{https://doi.org/10.48550/arXiv.1606.00199}
{doi:10.48550/arXiv.1606.00199}.}
\bibitem{DeyHouMandal19}
T. K. Dey, T. Hou, and S. Mandal,
Persistent 1-cycles: Definition, computation, and its application,
in \emph{Computational Topology in Image Context},
Lecture Notes in Computer Science \textbf{11382}, Springer, 2019,
pp. 123--136,
\href{https://doi.org/10.1007/978-3-030-10828-1_10}{doi:10.1007/978-3-030-10828-1\_10}.
\bibitem{BasuCox24}
S. Basu and N. Cox,
Harmonic persistent homology,
\emph{SIAM J. Appl. Algebra Geom.} \textbf{8} (2024), no.~1,
189--224,
\href{https://doi.org/10.1137/22M1518761}{doi:10.1137/22M1518761}.
\bibitem{HouParsaWang25}
T. Hou, S. Parsa, and B. Wang,
Tracking the persistence of harmonic chains: barcode and stability,
in \emph{41st International Symposium on Computational Geometry
(SoCG 2025)}, Leibniz Int. Proc. Inform. \textbf{332},
article~58, 2025,
\href{https://doi.org/10.4230/LIPIcs.SoCG.2025.58}
{doi:10.4230/LIPIcs.SoCG.2025.58}.
\bibitem{BarbensiEtAl22}
A. Barbensi, I. H. R. Yoon, C. D. Madsen, D. O. Ajayi,
M. P. H. Stumpf, and H. A. Harrington,
Hypergraphs for multiscale cycles in structured data,
\emph{arXiv:2210.07545} (2022),
\rev{\href{https://doi.org/10.48550/arXiv.2210.07545}
{doi:10.48550/arXiv.2210.07545}.}
\bibitem{YoonGhristGiusti23}
H. R. Yoon, R. Ghrist, and C. Giusti,
Persistent extensions and analogous bars: Data-induced relations
between persistence barcodes,
\emph{J. Appl. Comput. Topol.} \textbf{7} (2023), 571--617,
\href{https://doi.org/10.1007/s41468-023-00115-y}{doi:10.1007/s41468-023-00115-y}.
\bibitem{Manzanares26}
\rev{J. R. Manzanares,
Persistence pairing graphs: tracking changes of selected homology bases,
\emph{arXiv:2608.29719}, 2026,
\href{https://doi.org/10.48550/arXiv.2608.29719}
{doi:10.48550/arXiv.2608.29719}.}
\bibitem{GuzelKaygun22}
\.{I}. G\"uzel and A. Kaygun,
A new non-archimedean metric on persistent homology,
\emph{Comput. Stat.} \textbf{37} (2022), 1963--1983,
\href{https://doi.org/10.1007/s00180-021-01187-z}{doi:10.1007/s00180-021-01187-z}.
\bibitem{CarlssonMemoli10}
G. Carlsson and F. M\'emoli,
Characterization, stability and convergence of hierarchical clustering methods,
\emph{J. Mach. Learn. Res.} \textbf{11} (2010), 1425--1470,
\url{https://www.jmlr.org/papers/v11/carlsson10a.html}.
\bibitem{Thom54}
R. Thom, Quelques propri\'et\'es globales des vari\'et\'es
  diff\'erentiables,
\emph{Comment. Math. Helv.} \textbf{28} (1954), 17--86,
\href{https://doi.org/10.1007/BF02566923}{doi:10.1007/BF02566923}.
\bibitem{KreckLueck05}
M. Kreck and W. L\"uck, \emph{The Novikov Conjecture: Geometry and Algebra},
Oberwolfach Seminars, vol.~33, Birkh\"auser, Basel, 2005,
\href{https://doi.org/10.1007/b137100}{doi:10.1007/b137100}.
\bibitem{Rudyak08}
Y. B. Rudyak, \emph{On Thom Spectra, Orientability, and Cobordism},
corrected reprint, Springer Monographs in Mathematics, Springer, 2008,
\href{https://doi.org/10.1007/978-3-540-77751-9}{doi:10.1007/978-3-540-77751-9}.
\bibitem{ChambersEtAl10}
E. W. Chambers, V. de Silva, J. Erickson, and R. Ghrist,
Vietoris--Rips complexes of planar point sets,
\emph{Discrete Comput. Geom.} \textbf{44} (2010), 75--90,
\href{https://doi.org/10.1007/s00454-009-9209-8}
{doi:10.1007/s00454-009-9209-8}.
\bibitem{Oxley11}
J. Oxley, \emph{Matroid Theory}, 2nd ed., Oxford University Press, 2011.
\bibitem{BauerBotnanOppermannSteen20}
U. Bauer, M. B. Botnan, S. Oppermann, and J. Steen,
Cotorsion torsion triples and the representation theory of filtered
hierarchical clustering, \emph{Adv. Math.} \textbf{369} (2020), 107171,
\href{https://doi.org/10.1016/j.aim.2020.107171}{doi:10.1016/j.aim.2020.107171}.
\end{thebibliography}
\end{document}